\documentclass[a4paper,10pt]{amsart}

\usepackage{amssymb,amscd,amsthm,mathrsfs}
\usepackage[ansinew]{inputenc}
\usepackage[centertags]{amsmath}
\usepackage{graphicx,psfrag}

\newtheorem*{teo*}{Theorem}

\newtheorem*{moore}{Moore's Theorem}

\newtheorem{teo}{Theorem}

\newtheorem{cor}[teo]{Corollary}
\newtheorem{lema}{Lemma}[section]
\newtheorem{prop}[lema]{Proposition}
\newtheorem{teor}[lema]{Theorem}
\newtheorem{problem}{Problem}
\newtheorem{add}[lema]{Addendum}

\newcommand{\bi}{\begin{itemize}}
\newcommand{\ei}{\end{itemize}}

\theoremstyle{definition}

\theoremstyle{remark}

\newtheorem{obs}[teo]{Remark}

\numberwithin{equation}{section}

\newcommand{\ld}{\ensuremath{,\ldots,}}
\newcommand{\ssq}{\ensuremath{\subseteq}}
\newcommand{\smin}{\ensuremath{\setminus}}
\newcommand{\eps}{\ensuremath{\varepsilon}}

\newcommand{\twlog}{without loss of generality}

\newcommand{\T}{\ensuremath{\mathbb{T}}}

\newcommand{\N}{\ensuremath{\mathbb{N}}} 
\newcommand{\R}{\ensuremath{\mathbb{R}}}
\newcommand{\Z}{\ensuremath{\mathbb{Z}}}
\newcommand{\Q}{\ensuremath{\mathbb{Q}}}
\newcommand{\C}{\ensuremath{\mathbb{C}}}

\newcommand{\D}{\ensuremath{\mathbb{D}}}

\newcommand{\A}{\ensuremath{\mathbb{A}}}

\newcommand{\inte}{\ensuremath{\mathrm{int}}}

\newcommand{\kreis}{\ensuremath{\mathbb{T}^{1}}}

\newcommand{\torus}{\ensuremath{\mathbb{T}^2}}

\newcommand{\homeo}{\ensuremath{\mathrm{Homeo}}}
\newcommand{\homtwo}{\ensuremath{\mathrm{Homeo}_0(\mathbb{T}^2)}}

\newcommand{\alphlist}{\begin{list}{(\alph{enumi})}{\usecounter{enumi}\setlength{\parsep}{2pt}
      \setlength{\itemsep}{1pt} \setlength{\topsep}{5pt}
      \setlength{\partopsep}{3pt}}}

\newcommand{\arablist}{\begin{list}{\arabic{enumi}.}{\usecounter{enumi}\setlength{\parsep}{2pt}
          \setlength{\itemsep}{1pt} \setlength{\topsep}{5pt}
          \setlength{\partopsep}{3pt}}}

\newcommand{\romanlist}{\begin{list}{(\roman{enumi})}{\usecounter{enumi}\setlength{\parsep}{2pt}
              \setlength{\itemsep}{1pt} \setlength{\topsep}{5pt}
              \setlength{\partopsep}{3pt}}}

 \newcommand{\listend}{\end{list}}

\newcommand{\bulletlist}{\begin{list}{$\bullet$}{\setlength{\parsep}{2pt}
                \setlength{\itemsep}{1pt} \setlength{\topsep}{5pt}
                \setlength{\partopsep}{3pt}\setlength{\leftmargin}{15pt}}}

\newcommand{\foot}{\footnote}

\newcommand{\ncup}{\ensuremath{\bigcup_{n\in\N}}}

\newcommand{\nLim}{\ensuremath{\lim_{n\rightarrow\infty}}}
\newcommand{\iLim}{\ensuremath{\lim_{i\rightarrow\infty}}}

\newcommand{\kLim}{\ensuremath{\lim_{k\rightarrow\infty}}}

\newcommand{\filll}{\ensuremath{\mathrm{Fill}}}
\newcommand{\conn}{\ensuremath{\mathrm{Conn}}}
\newcommand{\cl}{\ensuremath{\mathrm{Cl}}}

\newcommand{\M}{\ensuremath{\mathcal{M}}}
\newcommand{\Acal}{\ensuremath{\mathcal{A}}}
\newcommand{\Ccal}{\ensuremath{\mathcal{C}}}
\newcommand{\U}{\ensuremath{\mathcal{U}}}

 \title{A classification of minimal
  sets of torus homeomorphisms}

\author{T.~J\"ager \and F.~Kwakkel \and A.~Passeggi}

\begin{document}

\begin{abstract} We provide a classification of minimal sets of
  homeomorphisms of the two-torus, in terms of the structure of their
  complement. We show that this structure is exactly one of the
  following types: (1) a disjoint union of topological disks, or (2) a
  disjoint union of essential annuli and topological disks, or (3) a
  disjoint union of one doubly essential component and bounded
  topological disks. Moreover, in case (1) bounded disks are
  non-periodic and in case (2) all disks are non-periodic.

  This result provides a framework for more detailed
  investigations, and additional information on the torus
  homeomorphism allows to draw further conclusions. In the
  non-wandering case, the classification can be significantly
  strengthened and we obtain that a minimal set other than the whole
  torus is either a periodic orbit, or the orbit of a periodic
  circloid, or the extension of a Cantor set. Further special cases
  are given by torus homeomorphisms homotopic to an Anosov, in which
  types 1 and 2 cannot occur, and the same holds for homeomorphisms
  homotopic to the identity with a rotation set which has non-empty
  interior.  If a non-wandering torus homeomorphism has a unique and
  totally irrational rotation vector, then any minimal set other than
  the whole torus has to be the extension of a Cantor set.\medskip

\end{abstract}


\maketitle

\section{Introduction and Statement of Results}

As minimal sets relate naturally to many other dynamical notions,
great effort has been devoted to the description of minimal sets and
their intrinsic structure, which led to the identification of
important subclasses like almost periodic or almost automorphic
minimal sets (see, for example, \cite{auslander, veech,shen/yi:1998}
and references therein). However, there exist only very few
situations in which a complete classification of the possible
structure of minimal sets in a given manifold is available. One of the
most important cases are orientation-preserving homeomorphisms of the
circle, whose minimal sets classify into either periodic orbits,
Cantor sets or the whole circle. By means of a suitable Poincar\'e
section, this also provides a classification of minimal sets of flows
on the two-torus generated by fixed point free vector fields, which are
a suspension of one of the three types occurring for circle homeomorphisms.
The Poincar\'e-Bendixon Theorem for planar flows or Aubry-Mather
Theory for twist maps provide further classical examples (see e.g.
\cite{katok/hasselblatt:1997}).  More recently, homeomorphisms of the
two-torus which are homotopic to the identity and have a single,
totally irrational rotation vector were studied in \cite{F}.  The
results in~\cite{F} include a classification of the minimal sets in
terms of the structure of their complement. For general surface
homeomorphisms, a more restricted classification is given in
\cite{nakayama} under the additional {\em a priori} assumption of
local connectedness. Here, our aim is to extend the main result in
\cite{F} to general homeomorphisms of the torus and to provide a
strengthened classification for non-wandering torus homeomorphisms.

Let $\T^2=\R^2/\Z^2$ denote the two-dimensional torus, $\pi \colon
\R^2\to\T^2$ the canonical projection and $\homeo(\T^2)$ the set of
homeomorphisms of $\T^2$. An open and connected set, respectively a
compact and connected set, in the plane $\R^2$ or torus $\T^2$ is
called a {\em domain}, respectively a {\em continuum}. We say an open
set $D \ssq \torus$ is a {\em topological disk} if it is homeomorphic
to $\D=\{z\in\C\mid |z|<1\}$ and call it {\em bounded}, if the
connected components of $\pi^{-1}(D)$ are bounded. Similarly, we say
an open set $A\ssq\torus$ is a {\em (topological) annulus} if it is
homeomorphic to the open annulus $\A=\kreis \times \R$ and call $A$
{\em essential} if it contains a closed curve which is homotopically
non-trivial in $\torus$. We call an open set $B\ssq \torus$ {\em
  doubly essential} if it contains two homotopically nontrivial curves
of different homotopy types. A subset $A$ of the torus (or any
surface) is called a {\em circloid}, if it is contained in an embedded
open annulus $\mathcal{A}$ and further (i) it is compact and
connected, (ii) its complement in $\mathcal{A}$ consists of exactly two connected
components $\U^-(A)$ and $\U^+(A)$ which are unbounded\foot{Here, we
  identify $\mathcal{A}$ with \A\ to define unboundedness.} below,
respectively above, and (iii) it is minimal with respect to inclusion
with properties (i) and (ii). A set which only satisfies (i) and (ii)
is called an {\em annular continuum}. The homotopy type of $A$ is
defined as the homotopy type of an essential loop in $\mathcal{A}$. We
call $A$ essential if this homotopy type is non-zero and homotopically
trivial otherwise. We call $A\ssq\torus$ {\em non-separating} if $A^c
= \T^2 \setminus A$ is connected.

Circloids and annular continua appear frequently in the theory of
torus and annular homeomorphisms
\cite{handel:1982,herman:1986,franks/lecalvez:2003,%
koropecki2009aperiodic,matsumoto:2011,jaeger:2009b}
and can be thought of as a generalisation of closed curves, adapted to
the needs of topological dynamics. When $f$ and
$\widetilde f$ are homeomorphisms of the two-torus with minimal sets
$\mathcal{M}$ and $\widetilde{\mathcal{M}}$, we say $(f,\mathcal{M})$
is an {\em extension} of $(\widetilde f,\widetilde{\mathcal{M}})$ if
there exists a continuous onto map $\Phi:\torus \to \torus$, homotopic
to the identity, which satisfies $\Phi\circ f = \widetilde f \circ
\Phi$ and $\Phi(\mathcal{M})=\widetilde{\mathcal{M}}$. When
$\widetilde{\mathcal{M}}$ is finite we simply say $\mathcal{M}$ is an
{\em extension of a periodic orbit}, when $\widetilde{\mathcal{M}}$ is
a Cantor set we say $\mathcal{M}$ is an {\em extension of a Cantor
  set}.

Given a connected component $U$ of $\mathcal{M}^c$ we say that it is
periodic if there exists $n\in\N$ such that $f^n(U)=U$, otherwise we
say that it is wandering. A minimal set of a homeomorphism $f$ of the
torus is a non-empty $f$-invariant compact set that is minimal,
relative to inclusion, with respect to the properties of being
$f$-invariant and compact. Our main result is the following.

\begin{teo}[Classification Theorem]\label{TEOA}
  Suppose $f\in\homeo(\torus)$ and $\mathcal{M} \neq \T^2$ is a
  minimal set. Then the complement of $\mathcal{M}$ consists of either:
\begin{enumerate}
\item[$(1)$] a disjoint union of topological disks.
\item[$(2)$] a disjoint union of at least one essential annulus and
  topological disks, where either:
\begin{itemize}
\item[\textup{(i)}] the essential annuli in $\mathcal{M}^c$ are
  periodic and $\mathcal{M}$ is the orbit of the boundary of an
  essential periodic circloid, or
\item[\textup{(ii)}] every connected component in $\mathcal{M}^c$ is
  wandering and $f$ is semi-conjugate to a one-dimensional irrational
  rotation,
\end{itemize}
\item[$(3)$] a disjoint union of exactly one doubly essential component and a number of
  bounded topological disks, where either:\begin{itemize}
\item[\textup{(i)}] $\mathcal{M}$ is an extension of a periodic
orbit, or
\item[\textup{(ii)}]  $\mathcal{M}$ is an extension of a Cantor set.
\end{itemize}
\end{enumerate}
Moreover, in case (1) bounded periodic disks cannot occur and in case
(2) only essential annuli can be periodic.
\end{teo}

We say a minimal set $\mathcal{M}$ is of type $N$ with $N=1,2,3$ if it
belongs to case $N$ in the above classification. This classification
provides a basic framework for a more precise study of the different
cases. In two important situations types 1 and 2 can be excluded. The
first is the case where $f\in\homeo(\torus)$ is homotopic to an Anosov
homeomorphism on \torus. Using classical results on Anosov
homeomorphisms~\cite{Bowen,manning:1974, walter} one obtains the
following.

\begin{cor}\label{l.anosov}
  Suppose $f\in\homeo(\torus)$ is homotopic to an Anosov
  homeomorphism. Then any minimal set of $f$ is of type 3.
\end{cor}

The second situation is more intricate and concerns the case where $f$
is homotopic to the identity. For such maps, an important topological
invariant is the rotation set given by
\begin{equation}
  \label{e.rotset}
  \rho(F) \ = \ \left\{\rho\in\R^2 \mid \exists z_i\in\R^2,\
   n_i\nearrow\infty : \iLim \left(F^{n_i}(z_i)-z_i\right)/n_i = \rho \right\} \ ,
\end{equation}
where $F:\R^2\to \R^2$ is a lift of $f$. This notion was introduced by
Misiurewicz and Ziemian, who showed that $\rho(F)$ is always a compact
and convex subset of the plane \cite{misiurewicz/ziemian:1989}.

\begin{cor}\label{t.non-empty-rotset}
  Suppose the rotation set of $f\in\homtwo$ has non-empty interior.
  Then any minimal set is of type 3.
\end{cor}

A result of Misiurewicz and Ziemian \cite{misiurewicz/ziemian:1991} states that
for all $\rho\in\inte(\rho(F))$ there exists a minimal set $M_\rho$ such that
$\rho$ is the unique rotation vector on $M_\rho$. In particular, there exist
uncountably many minimal sets. Corollary~\ref{t.non-empty-rotset} implies that for
all non-rational $\rho$ these are extensions of Cantor sets.

In the non-wandering case, a result of
Koropecki~\cite{koropecki2009aperiodic} on aperiodic invariant
continua of surface homeomorphisms allows to exclude unbounded disks
in type 1 of the Classification Theorem. This leads to the following
more restrictive classification. Recall that $f\in\homeo(\torus)$ is
called {\em non-wandering}, if there exist no wandering open sets.

\begin{teo}[Classification Theorem, non-wandering version]\label{teoa.nonwandering}
  Suppose $f\in\homeo(\torus)$ is non-wandering and $\M \neq \T^2$ is a minimal set. Then
  one of the following holds:
  \begin{itemize}
  \item[$(1^\textrm{nw})$] $\M$ is a periodic orbit;
  \item[$(2^\textrm{nw})$] $\M$ is the orbit of a periodic circloid;
  \item[$(3^\textrm{nw})$] $\M$ is the extension of a Cantor set, with all connected
    components non-separating.
  \end{itemize}
\end{teo}
Note that $(1^\textrm{nw})$ and $(3^\textrm{nw})$ belong to type $(3)$
in Theorem~\ref{TEOA}. $(2^\textrm{nw})$ belongs either to $(2)$ or
$(3)$, depending on whether the circloid is essential or not, since
the orbit of a homotopically trivial periodic circloid is a periodic
orbit extension.

Further information can be deduced if the rotation
set of $f\in\homtwo$ is reduced to a single point. In this case, we
call $f$ a {\em pseudo-rotation}.
\begin{cor}\label{t.ipr}
  Suppose $f$ is a non-wandering pseudo-rotation with rotation vector
  $\rho$ and $\mathcal{M} \neq \T^2$ is a minimal set.
  \begin{itemize}
  \item[(a)] If $\rho$ is totally irrational (its coordinates are
    rationally independent), then $\M$ is an extension of a Cantor set.
  \item[(b)] If $\rho$ is rational, then $\M$ is either an extension
    of a Cantor set, or the periodic orbit of either a point or a
    homotopically trivial circloid.
  \end{itemize}
\end{cor}

The paper is organised as follows. In Section~\ref{Preliminaries}, we
collect several preliminary topological results which will be used in
the later sections. In particular, we describe a procedure to fill in
subsets of the torus, similar to a standard construction in the plane.
Section~\ref{Classification} then contains the proof of the main
classification.  In Section~\ref{applications}, we consider several
special cases and applications of the classification. Finally, in
Section~\ref{Examples} we list and discuss a number of further
problems that naturally arise from the results in this paper.

\medskip

{\bf Acknowledgements.} We thank the referee for thoughtful comments
and suggestions on the manuscript.  We are indebted to Andres
Koropecki and Patrice Le Calvez for helpful comments and remarks. Our
results were first presented at the Visegrad Conference of Dynamical
Systems 2011 in Banska Bystrica, and we would like to thank the
organisers Roman Hric and Lubomir Snoha for creating this opportunity.
T.~J\"ager and A.~Passeggi acknowledge support by an
Emmy-Noether-grant Ja 1721/2-1 of the German Research Council.

\section{Fill-in Constructions}\label{Preliminaries}

In this section we collect several topological facts which are later
mixed with dynamical arguments to obtain our main results. In
particular, we describe a procedure to {\em `fill in'} subsets of the
torus which is similar to a standard construction in the plane, but
requires take care of some subtleties of surface topology. Even though
some of these constructions may be considered folklore, we therefore spell out
the details.

\subsection{Notation}\label{Basics}

Given a metric space $(X,d)$ and $C,D\subset X$, the Hausdorff distance is defined as
\begin{equation}
d_\mathcal{H}(C,D) = \max\{\sup_{x\in C} d(x,D),\sup_{y\in D}d(y,C)\}.
\end{equation}
The convergence of a sequence $\{C_n\}_{n\in \N}$ of subsets in $X$ to
$A\subset X$ in this distance is denoted either by
$C_n\rightarrow_{\mathcal{H}}A$ or by $\lim^\mathcal{H}_{n\to\infty}
C_n = A$. Note that $d_\mathcal{H}(C,D)<\eps$ if and only if $C\ssq
B_\eps(D)$ and $D\ssq B_\eps(C)$, and that the Hausdorff distance defines a metric if one restricts to compact subsets.

The fundamental group of $\torus$ will be denoted by $\pi_1(\torus)$.
Given a domain $U\ssq\torus$, consider the subgroup $G$ in
$\pi_1(\torus)$ given by classes of loops (i.e. simple closed curves)
in $U$. We say that $U$ is \textit{homotopically} \textit{trivial} if
$G=\{0\}$, \textit{essential} if $G$ is isomorphic to $\Z$ and
\textit{doubly essential} if $G$ is isomorphic to $\Z^2$. To an
essential set $U\subset \T^2$ we can associate a homotopy type given
by a vector $(p,q)\in\Z^2$ with $\textrm{gcd}(p,q)=1$, where $(p,q)$
is the generator of $G$. In this case, we call $U$ a $(p,q)$-essential
set. It is verified that an essential set $U$ is $(p,q)$-essential
if and only if every connected component $\widetilde{U}$ of its lift
satisfies $\widetilde U+(p,q)=\widetilde U$. We use $\conn(U)$ to
denote the set of connected components of $U\ssq \torus$ and
$\overline{\R^2}=\R^2\cup\{\infty\}$ to denote the Riemann sphere.  We
call a connected set $A \subset \T^2$ {\em bounded}, if all connected
components of $\pi^{-1}(A)$ are bounded. Note that this does not imply
a uniform bound on the size of the connected components of
$\pi^{-1}(A)$ (see Remark~\ref{r.boundedness}). 

Finally, given a homotopically trivial simple loop $\gamma\ssq\torus$, define the
embedded Jordan disk $B(\gamma):=\pi(B(\gamma_0)) \subset \T^2$, where
$\gamma_0\subset \R^2$ is a lift of $\gamma$ and $B(\gamma_0)$ the
Jordan disk bounded by $\gamma_0$.

\subsection{Fill-in of planar sets}

Given any connected set $A\ssq \R^2$, let $U_\infty(A)$ be the
connected component of $\overline{\R^2} \setminus A$ which contains
the point $\infty$.  The standard way to fill in the set $A$ is to
define
\begin{equation}
\filll_{\R^2}(A) \ = \ \R^2\smin U_\infty(A) \ .
\end{equation}

Note that when $\gamma$ is a loop in $\R^2$, then
$\filll_{\R^2}(\gamma)$ is just the closure of the Jordan domain of
$\gamma$, which will be denoted by $B(\gamma)$. Equivalent definitions
of $\filll_{\R^2}(A)$ are the following. First, if
$\{A_{\alpha}\}_{\alpha\in I}$ is the set of bounded connected
components of $\R^2\smin A$, then
\begin{equation}
\filll_{\R^2}(A):=A\cup \bigcup_{\alpha\in I}A_{\alpha}
\end{equation}
Since the union of a connected set with a connected component of its
complement is connected\foot{This is true in any $\sigma$-compact
  connected Hausdorff space.}, this allows to see in particular that
$\filll_{\R^2}(A)$ is always connected.  Secondly, if we say that a set
$A\ssq\R^2$ is {\em filled-in} if $\overline{\R^2}\smin A$ is
connected, then $\filll_{\R^2}(A)$ is just the smallest filled-in set
that contains $A$. For domains, a third equivalent characterisation is
given by the first part of the next statement.

\begin{lema}\label{l.fill-projection}
Let $A_0\ssq \R^2$ be open and connected. Then
\begin{equation}
  \filll_{\R^2}(A_0)=\{ z\in\torus\mid \exists\textrm{a loop}~\gamma\ssq A_0: z\in B(\gamma)\}.
\end{equation}
Further, for all $v\in\R^2$ we have that $A_0\cap(A_0+v) = \emptyset$
implies $\filll_{\R^2}(A_0)\cap (\filll_{\R^2}(A_0)+v) = \emptyset$.
\end{lema}

\begin{proof}
Let
\begin{equation}
\widetilde A_0 \ =\ \{z\in\R^2\mid \exists\textrm{a loop} ~ \gamma \ssq A_0: z\in B(\gamma)\} \ .
\end{equation}
Then $\widetilde A_0$ is simply-connected and therefore a topological
disk by the Riemann Mapping Theorem. In particular, it is filled-in.
Now, suppose $B$ is another filled-in set that contains $A_0$, but
does not contain $\widetilde A_0$. Then there is a loop in $A_0$ such
that its bounded component contains a point that is not in $B$.
However, this point cannot belong to the unbounded component of
$\R^2\smin A$, which is a contradiction. It follows that any filled-in
set that contains $A_0$ also contains $\widetilde A_0$, and therefore
$\widetilde A_0=\filll_{\R^2}(A_0)$. The second statement is a
consequence of the first statement.
\end{proof}

Our aim is now to define a similar $\filll$-operation for connected
subsets of the torus. This does not work for arbitrary connected
subsets of the torus (see the remarks at the end of this section), but
we show it does apply to subsets of $\torus$ which are either domains
or bounded continua. At the end of this section we collect several
basic results that will be used later in the proof of our main
results.

\subsection{Fill-in of domains in the torus}\label{filled-in}

We say that a domain $A \ssq\torus$ is {\em locally homotopically
  trivial} if every loop contained in $A$ which is homotopically
trivial in $\T^2$ is also homotopically trivial in $A$. Note that for
instance an essential annulus is locally homotopically trivial.

\begin{lema}\label{lem_simply_simply}
  If a domain $A \ssq\torus$ is locally homotopically trivial, then any connected
  component of $\pi^{-1}(A) \subset \R^2$ is simply connected.
\end{lema}

\begin{proof}
  Let $A_0 \subset \R^2$ be a connected component of $\pi^{-1}(A)$ and
  let $\gamma_0 \subset A_0$ be a simple closed curve. We have to show
  that $A_0$ contains $B(\gamma_0)$. Since $A_0$ is open, by
  approximating $\gamma_0$ by an analytic curve homotopic to
  $\gamma_0$, we may as well assume that that $\gamma :=
  \pi(\gamma_0)$ has finitely many self-intersections. Consequently,
  since $\gamma_0$ is compact, only finitely many other integer
  translates of $\gamma_0$ intersect $\gamma_0$, and the number of
  intersection points of $\gamma_0$ with the integer translates of
  $\gamma_0$ is finite.  Therefore, the intersection pattern produces
  a finite number of Jordan disks $J_1\ld J_n$ such that the
  boundary of $J_i$ is contained in $\pi^{-1}(\gamma)$, $J_i\cap
  \pi^{-1}(\gamma)=\emptyset$ and $\overline{B(\gamma_0)} =
  \bigcup_{i=1}^n \overline{J_i}$. Further, each of the disks $J_i$
  embeds injectively in $\torus$, since otherwise we would have an
  intersection $J_i\cap(J_i+v)$ for some $i\in\{1\ld n\}$, $v\in\Z^2$,
  contradicting the definition of the $J_i$. Since $A$ is simply
  connected and $\partial \pi(J_i)\ssq A$, we obtain $B(\pi(J_i))\ssq
  A$. However, this implies $J_i\ssq A_0$ for all $i=1\ld n$ and
  therefore $B(\gamma_0)\ssq A_0$.
\end{proof}

Locally homotopically trivial domains in the torus in the above sense can be classified as follows.

\begin{lema}[\cite{F}, Lemma 7]\label{Fe}
  Let $A\subset \T^2$ be open, trivial (respectively essential, doubly
  essential) and simply connected. Then $A$ is a disk (respectively
  essential annulus, $\T^2$).
\end{lema}

Given a domain $A\ssq \torus$ and a connected component $A_0$ of its lift, define
\begin{equation}
\filll(A) \ := \ \pi(\filll_{\R^2}(A_0)).
\end{equation}
Note that this definition does not depend on the choice of connected
component of $\pi^{-1}(A)$. Let us collect basic properties of the
$\filll$-operation for domains.

\begin{prop}[Fill-in of torus domains] \label{p.fill-properties}
Suppose $A\ssq\torus$ is a domain. Then the following hold.
\alphlist
\item[$(a)$] $\filll(A) = \{ z\in\torus\mid \exists\textrm{a trivial
    loop } \gamma\ssq A: z\in B(\gamma)\}$.
\item[$(b)$] $A$ is trivial and bounded iff $\filll(A)$ is a bounded disk.
\item[$(c)$] $A$ is trivial and unbounded iff $\filll(A)$ is an unbounded disk.
\item[$(d)$] $A$ is $(p,q)$-essential iff $\filll(A)$ is a $(p,q)$-annulus.
\item[$(e)$] $A$ is doubly essential iff $\filll(A)=\torus$.
\item[$(f)$] $\partial\filll(A)\ssq \partial A$.
\item[$(g)$] If $f\in\homeo(\torus)$, then $f(\filll(A)) = \filll(f(A))$.
\listend
\end{prop}

\begin{proof}
  The property $(a)$ is a direct consequence of
  Lemma~\ref{l.fill-projection} combined with
  Lemma~\ref{lem_simply_simply} applied to a connected component $A_0$
  of $\pi^{-1}(A)$. As a consequence, we obtain that $\filll(A)$ is
  simply-connected in $\torus$, such that according to Lemma~\ref{Fe}
  it is either a disk, an essential annulus or a doubly essential set.
  The properties $(b)$--$(e)$ therefore follow from the second part of
  Lemma~\ref{l.fill-projection}. 

  To prove $(f)$, assume for a contradiction that
  $z\in\partial\filll(A)$ but $z\notin\partial A$. Then, since
  $A\ssq\filll(A)$ is open, we have $z\notin \cl[A]$. Let $z_0$ be a
  lift of $z$ and denote by $C$ the connected component of $z_0$ in
  $\R^2\smin\pi^{-1}(\cl[A])$. Then either $C$ is contained in some
  integer translate of $\filll(A_0)$, but then
  $z\ssq\pi(C)\ssq\inte(\filll(A))$, or $C$ is disjoint from
  $\pi^{-1}(\filll(A))$, but then
  $z\in\pi(C)\ssq\inte(\torus\smin\filll(A))$. Hence, in both cases we arrive
  at a contradiction.
 
  Finally, in order to show $(g)$ let $F \colon \R^2\to\R^2$ be a lift of
$f\in\homeo(\torus)$. We then have
\begin{equation}
  f(\filll(A))=f\circ\pi(\filll_{\R^2}(A_0))=\pi\circ F(\filll_{\R^2}(A_0))
   \stackrel{(a)}{=}\pi(\filll_{\R^2}(F(A_0))) = \filll(A).
\end{equation}
This finishes the proof.
\end{proof}

\subsection{Fill-in of continua in the torus} \label{Doublyessential}

We now proceed to construct the fill-in of bounded continua in the
torus. As the following remark shows, some subtleties have to be
addressed and the construction does not work for general subsets of \torus. 

\begin{obs} \label{r.boundedness}

\indent (1) If $A \subset \T^2$ is a continuum, but not bounded, it
  is not clear how to define a fill-in. For example, consider the
  disjoint union of two essential loops $\gamma_1,\gamma_2\ssq\torus$
  with an infinite embedded line $\gamma\ssq\torus$ that accumulates
  on $\gamma_1$ in one and on $\gamma_2$ in the other direction.  Then
  the complement of any lift of $\gamma_1\cup \gamma_2\cup \gamma$ to
  $\R^2$ will have three connected components, and
  $\filll(\gamma_1\cup\gamma_2\cup\gamma)$ will depend on the
  particular choice and position of these components in the plane.

  \indent (2) A second problem comes from the fact that even if a
  connected subset of $A\ssq\torus$ is bounded in the sense of
  Section~\ref{Basics}, there is not necessarily a uniform bound on
  the diameter of the connected components of $\pi^{-1}(A)$.  Indeed,
  consider an irrational foliation of the torus given by the orbits of
  a Kronecker flow.  For each $n\in\N$ let $A_n$ be a segment of
  length $n$ in one of the leaves of this foliation, chosen such that
  no two segments are in the same leave.  Then for any $\eps>0$ there
  exists $n\in\N$ such that $A_n$ is $\eps$-dense. This implies that
  $A=\ncup A_n$ is connected: Suppose $\{U,V\}$ is an open disjoint
  cover of $A$. Then for $\eps>0$ sufficiently small both $U$ and $V$
  contain a disk of radius $\eps$, and consequently $A_n$ intersects
  both $U$ and $V$ when $n$ is large. Since $A_n\ssq A\ssq U\cup V$,
  this contradicts the connectedness of $A_n$. The connected
  components of $\pi^{-1}(A)$ are exactly the connected components of
  the lifts of the segments $A_n$. Hence, $A$ is bounded according to
  the above definition, but the connected components of $\pi^{-1}(A)$
  are not uniformly bounded in diameter.
 \end{obs} 

 In order to avoid problems, we restrict to bounded continua and
 first show that their complement is always doubly essential. This
 implies immediately that for any bounded continuum $A\ssq\torus$ there is
 only one connected component of $\pi^{-1}(A)$ up to translation by
 integer vectors. As a consequence, we will be able to define the
 fill-in in the same way as for domains.\medskip

In what follows, for a given family of pairwise disjoint sets
$\{X_n\}_{n\in\N}$ either in the plane or the torus, we denote its
union by $\biguplus_{n\in\N}X_n$.

\begin{lema}\label{pc0}
  Let $\{D_n\}_{n\in\N}$ be a family of pairwise disjoint bounded open
  disks in the plane. Then, $A=\R^2\smin\biguplus_{n\in\N}D_n$ is a
  connected set.
\end{lema}

\begin{proof}
  In order to prove the connectedness of $A$, we show that given any
  two points $z_0,z_1\in A$ there exists a connected subset of $A$
  containing $z_0$ and $z_1$. Denote the straight line segment from
  $z_0$ to $z_1$ by $S$. We assume without loss of generality that
  $S=[0,1]\times\{0\}$ and equip it with the canonical order on the
  unit interval. Further, define
\begin{equation}
  N=\{n\in\N\mid D_n\cap S\neq \emptyset\} ~
  \textup{and} ~ C=(S\cap A)\cup\bigcup_{n\in N}\partial D_n.
\end{equation}
We claim that $C$ is connected. In order to see this, suppose for a
contradiction that $U,V\ssq\R^2$ are disjoint open sets which both
intersect $C$ and whose union covers $C$. Suppose $z_0\in U$ and let
$z_1'\in C\cap V$. Further, let $z_-=\sup\{z\in S\cap A \cap U \mid
z\leq z_1'\}$. By compactness, $z_-\in S\cap A \cap V^c = S\cap A\cap
U$. Consequently, $z_-$ is the left endpoint of an interval $I=S\cap
D_n\ssq S\smin A$ for some $n\in\N,$ and the right endpoint $z_+$ of
this interval belongs to $S\cap A \cap V$. However, this means that
$U$ and $V$ both intersect $\partial D_n$, contradicting the
connectedness of $\partial D_n$.
\end{proof}

Using this together with Proposition~\ref{p.fill-properties}, we can
now show that compact and bounded subsets of the torus have doubly
essential complement.

\begin{lema} \label{l.compact_bounded} If $A\ssq\torus$ is compact and
  bounded, then $A^c$ is doubly essential. Consequently, if $A$ is
  connected, if $A$ is connected, all connected components of $\pi^{-1} (A)$ project injectively onto $A$
  and coincide up to translation by an integer vector.
  \end{lema}

\begin{proof}
  Suppose that $A^c$ is not doubly essential. Then due to
  Proposition~\ref{p.fill-properties}, for every $U\in\conn(A^c)$ the
  set $\filll(U)$ is either a bounded or unbounded topological disk or
  an essential annulus. We distinguish between three corresponding
  cases and show that in each of them $\pi^{-1}(A)$ contains an
  unbounded connected component, contradicting the boundedness of $A$.

First assume that there exists $U\in\conn(A^c)$ such that $\filll(U)$
is an essential annulus. Then $\partial \filll(U)$ consists of one or
two connected components, which are both contained in $A$ but at the
same time lift to unbounded components in $\R^2$. Secondly, suppose
there exists $U\in\conn(A^c)$ such that $D=\filll(U)$ is an unbounded
disk. Fix a connected component $D_0\ssq \pi^{-1}(D)$ and a point
$z_0 \in\partial D_0 \ssq \pi^{-1}(A)$. Then for any $N\geq 0$ we can
choose a sequence $\eta_n \ssq D_0$ of arcs of diameter $N$ converging
in Hausdorff topology and such that $\eta
:=\lim^\mathcal{H}_{n\to\infty}\eta_n$ contains $z_0$ and is contained
in $\partial D_0$. For example, we can identify $D_0$ with $\D$ by the
Riemann Mapping Theorem and choose the $\eta_n$ suitable segments in
the circles of radius $1-1/n$. Then $\eta$ is connected (as the
Hausdorff limit of connected sets) and of diameter $N$. Since $N$ was
arbitrary, this shows that the connected component of $z_0$ in
$\partial D_0\ssq \pi^{-1}(A)$ is unbounded. Finally, assume that all
connected components of $A^c$ are bounded. Then
\[ \widetilde{A}_0 = \R^2 \setminus \bigcup_{U \in \conn(\R^2
  \setminus \pi^{-1}(A))} \filll_{\R^2} (U) \] is the complement of a
family of bounded disks and therefore connected and unbounded by
Lemma~\ref{pc0}.

Thus, as claimed the complement of $A$ contains a doubly essential
component. This implies that $A$ is contained in a bounded topological
disk $D$. Every connected component $D_0$ of $\pi^{-1}(D)$ therefore
contains a subset $A_0$ that projects injectively onto $A$. Since $\pi:D_0\to D$ is 
a homeomorphism we obtain that $A_0$ is connected.
\end{proof}

Given a bounded continuum $A\ssq \torus$, we now choose an
arbitrary connected component $A_0$ of $\pi^{-1}(A)$ and let
\begin{equation}
\filll(A) \ := \ \pi(\filll_{\R^2}(A_0)) \ .
\end{equation} 
Since all connected components of $\pi^{-1}(A)$ coincide up to integer
translations, this definition does not depend on the choice of $A_0$.
Furthermore, we have the following.

\begin{lema} \label{fill2}
Suppose $A$ is a bounded continuum. Then
\begin{equation}
\filll(A)=A\cup \biguplus_{n\in\mathcal N} D_n(A),
\end{equation}
where the $D_n(A)$ are bounded disks and
\begin{equation}
\partial D_n(A)=A\cap \cl[D_n(A)]
\end{equation}
for every $n\in\mathcal N\ssq\N$.
\end{lema}

\begin{proof} Let $A_0$ be a connected component of $\pi^{-1}(A)$ and
$\{D_n(A_0)\}_{n\in\mathcal N}$ be the bounded connected components of
$\R^2\smin A_0$. As $A$ is connected, $D_n(A_0)$ is a disk for every
$n\in\mathcal N$. Further, we have
\begin{equation}
\partial D_n(A_0)=A_0\cap \cl[D_n(A_0)]
\end{equation}
and since $A_0$ is bounded we have that
\begin{equation}
D_n(A_0)\cap (D_n(A_0)+v)=\emptyset
\end{equation}
for every $v\in\Z^2\setminus \{0\}$. Consequently
$\pi(D_n(A_0))=D_n(A)$ is a bounded disk for every $n\in \mathcal N$.
We obtain that
\begin{equation}
  \filll(A)=\pi(\filll(A_0))=\pi(A_0\cup\biguplus_{n\in\mathcal N} D_n(A_0))
   =A\cup\biguplus_{n\in\mathcal N}D_n(A)
\end{equation}
and
\begin{equation}
  \partial D_n(A)=\partial \pi(D_n(A_0))=\pi (\partial D_n(A_0))
  = \pi(A_0\cap \cl[D_n(A_0)]) = A\cap \cl[D_n(A)].
\end{equation}
This finishes the proof.
\end{proof}

\section{Classification of minimal sets} \label{Classification}

In this section, we prove the main classification given in the introduction.

\subsection{Proof of the Classification Theorem}

We start the proof with the following trichotomy.

\begin{prop}\label{p.classification}
  Suppose $f\in\homeo(\torus)$ and $\mathcal M \neq \T^2$ is a minimal set. Then
  one of the following holds:
  \arablist
  \item[$(1)$] $\mathcal{M}^c$ is a disjoint union of topological disks.
  \item[$(2)$] $\mathcal{M}^c$ is a disjoint union of one or more essential annuli and
    topological disks.
  \item[$(3)$] $\mathcal{M}^c$ is a disjoint union of one double essential component and bounded
    topological disks.
  \listend
\end{prop}

\begin{proof}
Let $\conn_{T}(\mathcal{M}^c)\subset \conn(\mathcal{M}^c)$
be the trivial connected components and $\conn_{E}(\mathcal{M}^c)$ the
essential connected components of $\mathcal{M}^c$. Consider
\begin{equation}
\mathcal{M}':=\T^2\setminus \bigcup_{\Sigma\in \conn_T(\mathcal{M}^c)} \filll(\Sigma),
\end{equation}
which is compact and $f$-invariant by
Proposition~\ref{p.fill-properties}(g). We claim that
$\mathcal{M}'\cap \mathcal{M}\neq \emptyset$.  In order to see this,
let
\begin{equation}
\mathcal{F}=\{\filll(\Sigma)\mid \Sigma\in \conn_T(\mathcal{M}^c)\}
\end{equation}
and
\begin{equation}
  \mathcal{D}=\left\{\bigcup_{n\in\N} F_n \mid F_n\textrm{ is an increasing sequence in }
  \mathcal{F}\right\}.
\end{equation}
First, note that all elements in $\mathcal{F}$ are topological disks with boundary
contained in $\mathcal{M}$. Further, $\mathcal{F}$ is partially ordered by
inclusion and two elements of $\mathcal{F}$ are either disjoint or one is
contained in the other. Consequently, the same is true for
$\mathcal{D}$. Furthermore, let $\{F_n\}_{n\in\N}$ be an increasing sequence in
$\mathcal{F}$. Then $D = \bigcup_{n\in\N} F_n \in \mathcal{D}$ is a disk, and
\begin{equation}
{\lim}_{n\to\infty}^{\mathcal{H}} \partial F_n = \partial D \ssq\mathcal{M}.
\end{equation}
Now, suppose $\{D_n\}_{n\in\N}$ is an increasing sequence in $\mathcal{D}$. Then
$\widehat D=\bigcup_{n\in\N}D_n$ is again an element of
$\mathcal{D}$. Hence, if we define $\mathcal{D}_{\max}$
as the set of maximal elements of $\mathcal{D}$, then due to the Lemma of Zorn
every element of $\mathcal{D}$ is contained in an element of
$\mathcal{D}_{\max}$. Thus, we have
$$\mathcal{M}'\ =\ \T^2\setminus \bigcup_{D\in\mathcal{D}}D \ = \ \T^2\setminus
\biguplus_{D\in\mathcal{D}_{max}}D \ . $$ Since $\partial D\subset \mathcal{M}$
for every $D\in\mathcal{D}_{max}$, this implies in particular that
$\mathcal{M}'\cap \mathcal{M}\neq \emptyset$. By minimality of $\mathcal{M}$ we
therefore have $\mathcal{M}\subset\mathcal{M}'$, so that $\filll(\Sigma)=\Sigma$
for every $\Sigma\subset \conn_T(\mathcal{M}^c)$. In other words, every trivial
component of $\mathcal{M}^c$ is a disk.

There exists an integer vector $(p,q)\in\Z^2\setminus\{0\}$ such that every element in $\conn_E(\mathcal{M}^c)$
is $(p,q)$-essential. Moreover, the above argument adapted to this case shows
that for every $\Sigma \in \conn_E(\mathcal{M}^c)$ we have
$\filll(\Sigma)=\Sigma$, and hence every $\Sigma \in \conn_E(\mathcal{M}^c)$ is a
$(p,q)$-annulus.

To conclude the proof, it now suffices to remark that unbounded disks or essential
components cannot coexist with a doubly essential component, and that any doubly
essential component is necessarily unique.
\end{proof}

In order to show the non-existence of periodic bounded disks in
$\conn(\mathcal{M}^c)$, we start with a preliminary lemma.

\begin{lema}\label{lemanp}
  Let $\mathcal{M}\subset \T^2$ be a connected minimal set of a
  homeomorphism $f:\T^2\rightarrow \T^2$. Further, assume that there
  exists a periodic bounded disk $D_0\in \conn(\mathcal{M}^c)$ of
  period $p$. Then, $\mathcal{M}=\partial D_0= ... = \partial
  D_{p-1}$, where $D_k=f^k(D_0)$ for $k=1,..., p-1$.
\end{lema}

\begin{proof}
Since $\partial D_0\cup ... \cup \partial D_{p-1}$ is an $f$-invariant set
contained in $\mathcal{M}$, we have that
\[ \mathcal{M}=\partial D_0\cup ... \cup \partial D_{p-1}. \] Given
$x\in \mathcal{M}$, let $r(x)\in \{1,...,p\}$ be the number of disks
in $\{D_0, ..., D_{p-1}\}$ for which $x\in \cl[D_k]$ ($k=0,...,p-1$).
Now, for any $k_0=0\ld p-1$ the set $r^{-1}(\{k:k\geq k_0\})\ssq
\mathcal{M}$ is closed and invariant, and therefore either empty or
equal to $\mathcal{M}$.  By minimality of $\mathcal{M}$, this implies
that $r$ is constant, say $r=m\in\{1\ld p\}$.

Now, for every $x\in\mathcal{M}$ define $I_x=\bigcap_{i=1}^{m}\partial
D_{k_i}$, where $ D_{k_1}, ..., D_{k_m}$ are the disks for which $x\in
\cl[D_{k_i}]$.  For every $x\in\mathcal{M}$ the set $I_x$ is closed,
and the collection of these sets $\mathcal{Y}=\{I_x:x\in\mathcal{M}\}$
is a finite family. Suppose that it is given by $I_{x_1},...,I_{x_N}$.
Then, we have that $\mathcal{M}=\bigcup_{i=1}^N I_{x_i}$. Further,
$I_{x_i}\cap I_{x_j}=\emptyset$ if $i\neq j$, since $z\in I_{x_i}\cap
I_{x_j}$ would imply $r(z)>m$. Therefore, by connectedness of
$\mathcal{M}$ we have
\[ I_{x_1}=...=I_{x_N}=\mathcal{M}, \]
which implies that
\[ \mathcal{M}=\partial D_0= ... = \partial D_{p-1}. \]
This finishes the proof.
\end{proof}

\begin{lema}\label{np}
Let $\mathcal{M}\subset \T^2$ be a connected minimal set of a homeomorphism
$f:\T^2\rightarrow \T^2$, and assume that $\conn(\mathcal{M}^c)$ does not contain
any doubly essential component. Then every bounded disk $D_0\in
\conn(\mathcal{M}^c)$ is non periodic.
\end{lema}

\begin{proof} Let us suppose for a contradiction that there exists a
  periodic bounded disk $D_0\in \conn(\mathcal{M}^c)$. Since
  $\mathcal{M}$ is connected, Lemma \ref{lemanp} shows that
  $\mathcal{M}=\partial D_0$. Let $\widetilde{D}_0$ be a connected
  component of $\pi^{-1}(D_0)$. Then $\pi:\widetilde{D}_0\rightarrow
  D_0$ is a homeomorphism and $\pi:\partial \widetilde{D}_0\rightarrow
  \partial D_0$ is onto.  We now split the proof into two cases, each
  leading to a contradiction. First suppose
  \[ \left(\cl [\widetilde{D}_0]+v\right)\cap \cl[\widetilde{D}_0]=
  \emptyset ~\textup{for~every}~ v\in\Z^2\setminus \{0\}. \] Then
  $\cl[D_0]$ is bounded and compact, such that
  Lemma~\ref{l.compact_bounded} implies the existence of a doubly
  essential component in $D_0^c$.  Secondly, assume that
  \[ \left( \cl[\widetilde{D}_0]+v \right) \cap
  \cl[\widetilde{D}_0]\neq \emptyset
  ~\textup{for~some}~v\in\Z^2\setminus \{0\}.\] In this case, define
  $r:\mathcal{M}\rightarrow \Z$ as $r(x)=\# \{v\in \Z^2\setminus
  \{0\}:\widetilde{x}+v\in\partial\widetilde{D}_0\}$ where
  $\widetilde{x}\in\partial\widetilde{D}_0\cap \pi^{-1}(x)$. It is verified 
  that $r(x)$ does not depend on $\widetilde{x}$. Since
  $\widetilde{D}_0$ is bounded, $r$ is finite. Further, we have that
  $r^{-1}(\{k:k\geq k_0\})$ is a closed and $f$-invariant subset of
  $\mathcal{M}$ for every $k_0\in\Z$. This implies by minimality of
  $\mathcal{M}$ that $r^{-1}(\{k:k\geq k_0\})$ is either empty or
  equal to $\mathcal{M}$, so $r(x)$ does not depend on
  $x\in\mathcal{M}$. Therefore, $r(x)=m$ for some positive integer
  $m$. Define
  \[ \mathcal{Y}=\left \{ (v_1,...,v_m)\in (\Z^2)^m: \exists z \in
    \partial\widetilde{D}_0\mbox{ such that }z+v_1,...,z+v_m\in
    \partial\widetilde{D}_0 \right\}.\] Since $\widetilde{D}_0$ is
  bounded, the set $\mathcal{Y}$ has to be finite, say
  $\mathcal{Y}=\{\xi_1,...,\xi_N\}$ with $\xi_i=(v^i_1\ld v^i_m)$. For
  $k=1,...,N$ define the sets
  \[ A_k=\{z\in
  \partial\widetilde{D}_0:z+v^k_1,...,z+v^k_m\in\partial\widetilde{D}_0\}.\]
  It is readily verified that $A_k\subset \partial\widetilde{D}_0$ is
  closed and that $\partial\widetilde D_0=\bigcup_{k=1}^N A_k$.
  
  When $\xi_i$ is just a permutation of the vector $\xi_j$, then
  obviously $A_i=A_j$. Otherwise, we must have $A_i\cap A_j =
  \emptyset$, since in this case the value $r(z)$ would be strictly
  greater than $m$ for any $z\in A_i\cap A_j$, which is not possible.
  Therefore the sets $A_i$ are either equal or pairwise disjoint. As
  $\partial\widetilde{D}_0$ is connected all sets $A_i$ have to
  coincide, and this implies $\partial\widetilde D_0=A_1$.  However,
  this means that $z+nv^1_j\in\partial\widetilde{D}_0$ for every
  $n\in\N$ and $j=1\ld m$, contradicting the boundedness of
  $\widetilde{D}_0$.
\end{proof}

\begin{prop}\label{p.periodic-nonexistence}
  If $\mathcal{M}$ in Proposition~\ref{p.classification} is of type 1 or 2, then
  $\conn(\mathcal{M}^c)$ does not contain any bounded periodic disk.
\end{prop}

\begin{proof} Suppose for a contradiction that $\mathcal{M}$ is of type 1 or 2 and
$D_0\in\conn(\mathcal{M}^c)$ is a bounded periodic disk of period $p$. Let
$D_k=f^k(D_0)$ as before. Then $\bigcup_{k=0}^{p-1}\partial D_k \ssq\mathcal{M}$
is compact and invariant, so that by minimality $\bigcup_{k=0}^{p-1}\partial
D_k =\mathcal{M}$.

Let $\Lambda_0$ be the connected component of $\mathcal{M}$ which contains
$\partial D_0$. Then $\Lambda_0$ is $q$-periodic for some $q\leq p$ and
minimal for $f^q$. By Lemma~\ref{np}, $\Lambda_0^c$ contains a doubly
essential component, and so does $f^i(\Lambda_0)^c$ for $i=0\ld q-1$. However,
due to Lemma~\ref{l.compact_bounded} this implies that
\begin{equation}
\mathcal{M}^c=\left(\bigcup_{i=0}^{q-1}f^i(\Lambda_0)\right)^c
\end{equation}
contains a doubly essential component, contradicting our assumption.
\end{proof}

In the next section we will see that also unbounded disk are wandering for type 2 minimal sets.

\subsection{Minimal sets of type 2}

In this section, we give a more detailed description of minimal sets
of type 2. Our aim is the following addendum to
Proposition~\ref{p.classification}.

\begin{add}\label{a.type2}
  Suppose $f\in\homeo(\torus)$ and $\mathcal{M}$ is a minimal set of
  type 2. Then one of the following holds.  \romanlist
\item[\textup{(i)}] The essential annuli in $\textrm{Conn}(\mathcal{M}^c)$ are
  periodic and $\mathcal{M}$ is the orbit of the boundary of an
  essential periodic circloid. Further any disk in $\textrm{Conn}(\mathcal{M}^c)$ is wandering.
\item[\textup{(ii)}] $f$ is semiconjugate to a one-dimensional irrational
  rotation, and every element in $\textrm{Conn}(\mathcal{M}^c)$ is wandering.  \listend
\end{add}
We start with some purely topological facts concerning circloids.  We
call a set $A\ssq \A$ {\em essential}, if $\A\smin A$ does not contain
a connected component which is unbounded above and below. If $A$ is
bounded above, we denote by $\U ^+(A)$ the connected component of
$\A\smin \cl(A)$ which is unbounded above. Similarly, we define
$\U^-(A)$ when $A$ is bounded below. Further, we write $\U^{-+}(A)$
instead of $\U^+(\U^-(A))$, and use analogous notation for longer
concatenations of these operations.  This leads to a simple procedure
to produce circloids.

\begin{lema}[\cite{jaeger:2009b}] \label{l.frontiers}
  Suppose $A\ssq \A$ is essential and bounded above. Then
\begin{equation}
\mathcal{C}^+(A) \ = \ \A \smin (\U^{+-}(A) \cup \U^{+-+}(A))
\end{equation}
is a circloid. Further $\U^{+-+-}(A)=\U^{+-}(A)$.
\end{lema}

We call $\mathcal{C}^+(A)$ the {\em upper frontier} of $A$, and similarly one
can define a {\em lower frontier} $\mathcal{C}^-(A)$.

\begin{lema} \label{l.boundary-frontier}
Under the assumptions of Lemma~\ref{l.frontiers}, we have that
\begin{equation}
  \partial\mathcal{C}^+(A) \ \ssq \ \partial A.
\end{equation}
In particular, any essential continuum in $\A$ contains the boundary
of an essential circloid.
\end{lema}

\begin{proof}
  In general, when $S\ssq\A$ is essential and bounded above we have
  $\partial \U^+(S)\ssq\partial S$, and the analogous statement holds
  if $S$ is bounded below. Applying this several times, we obtain
\begin{equation}
\partial \U^{+-+}(A) \ \ssq\  \partial \U^{+-}(A) \ \ssq \  \partial \U^+(A) \ \ssq \ \partial A.
\end{equation}
Furthermore, $E^+=\partial \U^{+-+}(A)$ is an essential continuum
which is disjoint from $\U^{+-}(A)$ and $\U^{+-+}(A)$ and therefore
contained in $\mathcal{C}^+(A)$. Consequently
$C^+=\A\smin(\U^-(E^+)\cup \U^+(E^+))$ is an annular continuum
contained in $\Ccal^+(A)$, and by minimality of the latter we obtain
$C^+=\Ccal^+(A)$. However, this means that
\begin{eqnarray*}
\partial\Ccal^+(A) & = & \partial C^+ \ = \ \partial (\A\smin\partial C^+) \\ & = &
\partial(\U^-(E^+)\cup\U^+(E^+))) \ \ssq \ \partial E^+ \ = \ \partial
\U^{+-+}(A) \ \ssq \ \partial A,
\end{eqnarray*}
as required.
\end{proof}

Given two essential continua $E_1$ and $E_2$, we write $E_1\prec E_2$
if $E_1\ssq \mathcal{U}^-(E_2)$. We say that a sequence of essential
continua $\{E_n\}_{n\in\N}\subset\A$ is bounded if there exist two
essential continua $E,F\subset \A$ such that $E\prec E_n\prec F$ for
every $n\in\N$.

\begin{lema}
  \label{l.increasing-essentials-limit}
  Suppose $\{E_n\}_{n\in\N}\ssq\A$ is a bounded sequence of essential continua with $E_n\prec
  E_{n+1}$ for all $n\in\N$. Then the $E_n$ converge in Hausdorff limit to the
  essential continuum $\partial \U^-$, where $\U^-=\bigcup_{n\in\N} \U^-(E_n)$.
\end{lema}

\begin{proof}
We have to show that
\begin{equation} \label{e.essential-limit} \forall \eps>0 \ \exists n_0\in\N \
  \forall n\geq n_0 \ : \quad \partial \U^- \ssq B_\eps(E_n) \textrm{ and } E_n
  \ssq B_\eps(\partial \U^-) \ .
\end{equation}
Note that for all $n\in\N$ the set $E_n$ is contained in $\U^-(E_k)$ for all
$k>n$. Conversely, $\partial \U^-\ssq\U^+(E_n)$ for all $n\in\N$,
since the bounded connected components of $\A\smin E_n$ are all contained in
$\U^-(E_{n+1})\ssq \U^-$ and can therefore not intersect $\partial \U^-$.

Now, first assume that there exist infinitely many $n\in\N$ with $\partial \U^-
\nsubseteq B_\eps(E_n)$. Choose a sequence $n_i\nearrow \infty$ and
$z_i\in\partial \U^-$ with $z_i\notin B_\eps(E_{n_i})$. Note that this implies
$z_i\notin B_\eps(E_{n_j})$ for all $j\leq i$, since the straight arc from $z_i$
to the nearest point in $E_{n_j}$ first has to pass through $E_{n_i}$. By
compactness, we may assume that the limit $z=\iLim z_i\in\partial \U^-$
exists. Then $B_\eps(z)\cap E_n=\emptyset$ for all $n\in\N$.  Obviously
$B_\eps(z)$ cannot be contained in $\U^-(E_n)$ for any $n\in\N$. Further,
$B_\eps(z)$ can also not be contained in a bounded component of $\A\smin E_n$,
since it would then be contained in $\U^-(E_{n+1})$. Consequently $B_\eps(z)\ssq
\U^+(E_n)$ for all $n\in\N$. However, this means that $\U^-$ does not intersect
$B_\eps(z)$, contradicting $z\in\partial \U^-$.

Conversely, suppose $E_n\nsubseteq B_\eps(\partial \U^-)$ for
infinitely many $n\in\N$. Choose $n_i\nearrow\infty$ and $z_i\in
E_n\smin B_\eps(\partial \U^-)$ so that the limit $z=\iLim z_i$
exists. Then on the one hand we have $z\notin B_{\eps/2}(\partial
\U^-)$, but on the other hand $z$ is a limit point of points $z_i\in
E_n\ssq \U^-(E_{n+1})\ssq \U^-$, a contradiction.  This shows that
(\ref{e.essential-limit}) holds and thus
$\lim^\mathcal{H}_{n\to\infty} E_n=\partial \U^-$ as claimed.
\end{proof}

We now turn to minimal sets of type 2, starting with a simple
observation.

\begin{lema}\label{l.eigenvector-of-homology}
  Suppose \M\ is a minimal set of $f\in\homeo(\torus)$ and $\M^c$ contains an
  essential annulus $\Acal$ of homotopy type $(p,q)$.  Then $(p,q)$ is an
  eigenvector of the induced action $f_*$ on homotopy.
\end{lema}

In fact, the assertion of the lemma is true for any annulus which is
either invariant or disjoint from its image.

\begin{proof}[Proof of Lemma~\ref{l.eigenvector-of-homology}]
  When $\Acal$ is invariant, then the fact that its homotopy vector is
  preserved is obvious. When $f(\Acal)$ and $\Acal$ are disjoint, this
  follows from the fact that essential annuli of different homotopy
  types have to intersect.
\end{proof}

Now, choose $A\in\textrm{SL}(2,\Z)$ such that $(p,q)=A\cdot(1,0)^t$
and let $f_A$ be the torus homeomorphism induced by $A$. Then $(1,0)$
is an eigenvector of the action on homotopy of
\begin{equation}\label{e.fhat}
  \widehat f\ =\ f_A^{-1}\circ
  f\circ f_A \ ,
\end{equation}
and this implies that there exists a lift $\widetilde f : \A\to\A$
which projects to $\widehat f$ under the canonical projection $\pi_\A
: \A \to \torus$. We either have $\widetilde f(z+(0,1))=\widetilde
f(z)+(0,1)$ or $\widetilde f(z+(0,1))=\widetilde f(z)-(0,1)$.  We call
$\widetilde f$ {\em order-preserving} in the first case and {\em
  order-reversing} in the second. When $\widetilde f$ is
order-preserving, we define the {\em rotation interval of $f$
  orthogonal to $(p,q)$} by

\begin{equation}
  \label{e.rot-interval}
  \rho_{(p,q)}(f) \ = \ \frac{1}{\|(p,q)\|_2}\cdot\left\{
    \rho\in\R \ \left| \  \exists n_i\nearrow\infty,\
      z_i\in\A : \iLim \left(\pi_2\circ \widetilde f^{n_i}(z_i)-\pi_2(z_i)\right)/n_i
      = \rho \right.\right\} \ .
\end{equation}
Of course, due to the freedom in the choice of the lift $\widetilde f$
the interval $\rho_{(p,q)}$ is only well-defined up to translation by
integer multiples of $\|(p,q)\|_2$, and we will implicitly understand
it in this sense. Note that when $f$ is homotopic to the identity,
then $\rho_{(1,0)}(\widetilde f)$ is just the projection of $\rho(F)$
to the second coordinate. In general, it is the projection of
$\rho(F)$ to the line $(-q,p)\cdot \Z$.

In the order-reversing case, we apply the above definition to $f^2$
and let $\rho_{(p,q)}(f) = \rho_{(p,q)}(f^2)/2$. In this case, we have
\begin{lema}
  If $\widetilde f$ is order-reversing, then $\rho_{(p,q)}(f)$ contains $0$.
\end{lema}

\begin{proof} For every $n\in 2\Z+1$ the map $\widetilde f^n$ reverses
  orientation, so that $D(k) = \pi_2\circ \widetilde f^n(0,k)-k$
  goes to $\pm\infty$ as $k$ goes to $\mp\infty$. Consequently, for
  sufficiently large $k$ the numbers $D(k)$ and $D(-k)$ have opposite
  sign. Therefore, by the Intermediate Value Theorem any arc joining
  $(0,k)$ to $(0,-k)$ contains a point with $\pi_2\circ \widetilde
  \pi_2\circ f^n(z)-\pi_2(z)=0$.
\end{proof}

In the same way, it is shown that $\rho_{(p,q)}(f)$ is
connected and, in the order-reversing case, symmetric around 0. In the
situation we consider, the rotation interval is degenerate.
\begin{lema}\label{l.rational}
  \label{l.degenerate-rotinterval} Let $f\in\homeo(\torus)$ and
  suppose there exists an annulus $\Acal\ssq\torus$ of homotopy type
  $(p,q)$ which is either periodic or wandering. Then
  $\rho_{(p,q)}(f)$ contains a single number. If $\Acal$ is periodic,
  this number is rational.

  In particular, suppose $f\in\homtwo$ has a periodic or wandering
  annulus.  Then the rotation set $\rho(F)$ is contained in a rational
  line.
\end{lema}

This is a direct corollary to \cite[Lemma 1.4]{koropecki:2007}, and we
omit the simple proof. From now on, we identify $\rho_{(p,q)}(f)$ with
the unique real number $\rho$ it contains and call it the {\em
  rotation number of $f$ orthogonal to $(p,q)$}. The following lemma
deals with the case where this rotation number is irrational. We omit
the proof, which can be found in \cite{koropecki:2007}. The author
uses an additional minimality assumption, but this is actually not
needed. Alternatively, the result also follows from a minor
modification of \cite[Proof of Theorem~C]{jaeger:2009b}.

\begin{lema} \label{l.semi-conjugacy} Suppose $f\in\homeo(\torus)$ has a
  wandering annulus $\Acal$ of homotopy type $(p,q)$ and $\rho_{(p,q)}(f)$ is
  irrational. Then $f$ is semiconjugate to the corresponding irrational
  rotation on \kreis.
\end{lema}

In order to treat the rational case, we first need some more information
concerning circloids.
\begin{lema} \label{l.bounded-circloid} Let $f\in\homeo(\torus)$ and
  suppose there exists a wandering circloid $C \ssq\torus$ of homotopy
  type $(p,q)$. Further, assume $\rho_{(p,q)}(f)=0$ and let
  $C_0\ssq\A$ be a lift of $f_A(C)$, where $A\in\textrm{SL}(2,\Z)$ is
  chosen as in (\ref{e.fhat}). Then $C_0-(0,1) \prec
  \widetilde f^{2n}(C_0)\prec C_0+(0,1)$ for all $n\in\N$. In
  particular, the orbit of $C_0$ under $\widetilde f$ is bounded.
\end{lema}

\begin{proof}
  Since $C$ is wandering, $\widetilde f^{2n}(C)$ is disjoint from
  $C_0+(0,1)\cdot\Z$ for all $n\neq 0$. Suppose for a contradiction
  that $\widetilde f^{2n}(C_0)$ does not lie between $C_0-(0,1)$ and
  $C_0+(0,1)$ for some $n\in\N$, for example $\widetilde f^n(C_0)
  \prec C_0+(0,1).$ Then, by induction $\widetilde f^{i2n}\prec
  C_0+(0,i)$. This, however, implies that the rotation number is
  strictly positive, contradicting the assumptions.
\end{proof}

\begin{lema} \label{l.essential-limit} Let $\widetilde f\in\homeo(\A)$
  and suppose $E$ is an essential continuum which is disjoint from its
  image and has a bounded orbit. If $\widetilde f$ is
  order-preserving, then $\lim^\mathcal{H}_{n\to\infty} \widetilde
  f^n(E)$ exists and contains an invariant circloid. If $\widetilde f$
  is order-reversing, then $\lim^\mathcal{H}_{n\to\infty}\widetilde
  f^{2n}(E)$ exists and contains a circloid which is either invariant
  or two-periodic.
\end{lema}

\begin{proof} It suffices to treat the order-preserving case, since we only have to
consider $\widetilde f^2$ when $\widetilde f$ reverses order. We either have $E\prec
\widetilde f(E)$ or $E\succ \widetilde f(E)$. We treat the first case, the other one is
similar.

If we let $E_n:=\widetilde f^n(E)$, then this is an increasing sequence with respect
to $\prec$ and converges in Hausdorff distance to the essential continuum
$\partial U^-$ given by Lemma~\ref{l.increasing-essentials-limit}. Since
$\partial U^-=\lim^\mathcal{H}_{n\to\infty} \widetilde f^n(E)$ is contained in \M\, the statement
follows from Lemma~\ref{l.boundary-frontier}.
\end{proof}

We now turn to the proof of the fact that only essential annuli can be
periodic connected components of the complement of a type two minimal
set. We start with the following well-known fact.

\begin{lema}\label{addedlemma}
  Let $f\in\homeo(\T^2)$ and suppose $\mathcal{A} \subset \T^2$ is an
  $f$-invariant essential annulus. Then, for every essential simple loop
  $\gamma$ in $\mathcal{A}$ and any neighbourhood $V\ssq\mathcal{A}$
  of $\gamma$ there exists $g\in\homeo(\T^2)$ such that:
\begin{itemize}
\item[(i)] $g$ is homotopic to $f$;
\item[(ii)] $g(\gamma)=\gamma$;
\item[(iii)] $f(x)=g(x)$ for every $x\in V^c$.
\end{itemize}

\end{lema}

To a homeomorphism $g$ as in the lemma, we can naturally associate a
homeomorphism $\overline{g}:\mathbb{S}^2\rightarrow \mathbb{S}^2$,
where $\mathbb{S}^2$ is the two-dimensional sphere by cutting the
torus open along $\gamma$ to obtain an open annulus
$\torus\smin\gamma$ and then compactifying this annulus by adding two
points $N$ and $S$. Then $g|_{\gamma^c}$ is conjugate to
$\overline{g}|_{\{N,S\}^c}$ by a semiconjugacy
$h:\torus\smin\gamma\to\mathbb{S}^2\smin\{N,S\}$ and
$\overline{g}(\{N,S\})=\{N,S\}$.  Furthermore, $h$ maps the two
components of $\textrm{Cl}[\mathcal{A}]\setminus \gamma$ to two
different components $U_1$ and $U_2$ in $S^2$ with
$\textrm{Cl}[U_1]=U_1\cup \{N\}$ and $\textrm{Cl}[U_2]=U_2\cup \{S\}$.
The advantage that this transformation to a sphere homeomorphism has,
is that it allows to apply the following theorem by Matsumoto and
Nakayama \cite{MaNa}.

\begin{teo}\label{MaNa}
  Let $\overline{g}:S^2\rightarrow S^2$ be a homeomorphism and
  $C\subset S^2$ be a non singleton compact and connected set. Further
  assume that $C$ is a minimal set of $\overline{g}$. Then, there are
  exactly two periodic connected components $A_1$ and $A_2$ in $C^c$.
\end{teo}

In our context, it is obvious from the invariance of $\mathcal{A}$
that $A_1$ and $A_2$ are the images of the two components of
$\Acal\smin\gamma$ under $h$. We obtain the following.
\begin{prop}\label{addedprop}
  Suppose that $\mathcal{M}$ is a type 2 minimal set for
  $f\in\homeo(\T^2)$. Further suppose there exists an essential
  annulus $\mathcal{A}$ in $\textrm{Conn}(\mathcal{M}^c)$ which is periodic.
  Then any disk in $\textrm{Conn}(\mathcal{M}^c)$ is wandering.
\end{prop}
\begin{proof}

  We denote by $C_1,C_2\subset \T^2$ the two connected components of
  $\partial\mathcal{A}$, allowing for $C_1=C_2$ in case
  $\partial\Acal$ is connected. Then due to the fact that
  $\mathcal{A}$ is periodic there exists $n\in\N$ such that $C_1$ and
  $C_2$ are minimal sets of $f^n$ and $\mathcal{M}=(C_1\cup
  C_2)\cup...\cup f^n(C_1\cup C_2)$. Hence, given a disk
  $V\in\textrm{Conn}(\mathcal{M}^c)$, then since $\partial V$ is a
  connected set contained in $\mathcal{M}$ there exists $n_0 \in \N$
  such that $V_1:=f^{n_0}(\partial V)\subset C_1\cup C_2$. We assume
  \twlog\ that $n_0=0$ and $V_1\ssq C_1$. This means that $V$
  is a connected component of $C_1^c$. However, if we consider
  $g\in\homeo(\T^2)$ given by Lemma \ref{addedlemma} applied to $f^n$
  and some essential loop $\gamma\ssq\Acal$, we have that $C_1$ is a
  minimal set for $g$.

  Suppose for a contradiction that $V$ is periodic by $f$. Then $V$ is
  a periodic connected component of $C_1^c$ for $g$. However, as $V$
  is a disk it cannot coincide with one of the two periodic components
  $A_1$ and $A_2$ that $g$ admits. This contradicts
  Theorem~\ref{MaNa}.
\end{proof}

We are ready now to give the proof of the Addendum~\ref{a.type2}.

\begin{proof}[Proof of Addendum~\ref{a.type2}] Suppose
  $f\in\homeo(\torus)$ and $\M$ is a minimal set of type 2. Let
  $\Acal$ be an essential annulus of homotopy type $(p,q)$ in
  $\conn(\M^c)$. Then $\Acal$ is either wandering or periodic, and in
  each case $\rho_{(p,q)}(\widetilde f)$ contains a unique number
  $\rho\in\R$ by Lemma~\ref{l.rational}. If $\rho$ is irrational, then
  $\Acal$ is wandering by Lemma~\ref{l.degenerate-rotinterval}, and
  Lemma~\ref{l.semi-conjugacy} provides the existence of a
  semiconjugacy to an irrational rotation of \kreis. Furthermore, due
  to the existence of such a semi-conjugacy any element in
  $\textrm{Conn}(\mathcal{M}^c)$ is wandering. Thus, we are in case
  (ii) of the addendum.

Now, assume $\rho$ is rational. Passing to an iterate $f^k$ and
choosing the right lift $\widetilde f$ of $\widehat f^k$ in
(\ref{e.fhat}), we may assume \twlog\ that $\rho_{(p,q)}(f^k)=0$. Let
$\Acal_0$ be a lift of $f_A(\Acal)$, where $A\in\textrm{SL}(2,\Z)$ is
chosen as in (\ref{e.fhat}), and let $\widetilde f$ be the lift of $f^k$
used to compute the rotation interval of $f^k$. Then either by
invariance or by Lemma~\ref{l.bounded-circloid}, the orbit of
$C_0=\Ccal^+(\Acal_0)$ under $\widetilde f$ is bounded. Hence, by
Lemma~\ref{l.essential-limit}, $\lim^\mathcal{H}_{n\to\infty}
\widetilde f^{2n}(C_0)$ contains an $\widetilde f^2$-invariant
circloid $\widetilde C$. Since $\widetilde C$ is disjoint from
$\Acal_0+(0,1)\cdot\Z$, it projects to a circloid $C$ on $\torus$
which is $2k$-periodic under $f$. Furthermore, $C$ is contained in the
Hausdorff limit of $f^{2kn}(\partial A)$ and thus in \M. By
minimality, we obtain $\M=\bigcup_{n=1}^{2k}f^n(C)$. Moreover, in this
case Proposition \ref{addedprop} implies that any disk in
$\textrm{Conn}(\mathcal{M}^c)$ has to be wandering, which means that
we are in case (i) of the addendum.
\end{proof}

\begin{obs}
  The above proof shows that case (i) of the addendum corresponds exactly to a
  rational rotation number orthogonal to the homotopy vector of the essential
  annuli, whereas case (ii) corresponds to an irrational rotation number.
\end{obs}

\subsection{Minimal sets of type 3}

The following addendum to Proposition~\ref{p.classification}
concerning the structure of minimal sets of type 3 completes the proof
of Theorem~\ref{TEOA}.

\begin{add}\label{Cantorextension}
  Suppose $f\in\homeo(\torus)$ and $\mathcal{M}$ is a minimal set of
  type 3. Then $\mathcal{M}$ is an extension of either a periodic
  orbit or a Cantor set.
\end{add}

Again, we first recall some purely topological facts. We say
$U\ssq\torus$ is {\em non-separating} if $\torus\smin U$ is connected.
We call a partition into continua $\mathcal{U}=\{U_i\}_{i\in I}$ of $\T^2$ an {\em
  upper semi-continuous decomposition} if it satisfies
\begin{itemize}
\item [(i)] $\biguplus_{i\in I} U_i = \torus$ and $U_i \cap U_j = \emptyset$ if $i \neq j$,
\item[(ii)] $U_i$ is a compact, bounded and non-separating set for every
$i\in I$;
\item[(iii)] if $\{U_n\}_{n\in \N}\subset \mathcal{U}$ has Hausdorff
limit $\mathcal{C}$, then there exists $U_0\in\mathcal{U}$ so that
$\mathcal{C}\subset U_0$ (upper semi-continuity property).
\end{itemize}
Further, we say  that a map
$\Phi:\T^2\rightarrow \T^2$ is a Moore projection for the decomposition $\mathcal{U}$
if it satisfies

\begin{itemize}
 \item[(i)] $\Phi$ is continuous and surjective;
 \item[(i)] $\Phi$ is homotopic to the identity;
 \item[(ii)] $\Phi^{-1}(x)\in \mathcal{U}$ for all $x\in\T^2$.
\end{itemize}

Now, the following is a classical decomposition theorem by R. Moore (see e.g.\cite{whyburn}).

\begin{moore}\label{moorteo}
For any upper semi-continuous decomposition of $\T^2$ there exists a Moore
projection.
\end{moore}
For our purposes, we have to ensure that under suitable conditions this
projection produces a Cantor set.

\begin{lema}\label{disc}
Let $A\subset\T^2$ be closed and denote its connected components by
$\{A_i\}_{i\in I}$. Suppose that the decomposition
$\mathcal{U}=\{A_i\}_{i\in I}\cup \{\{x\}:x\notin \bigcup_{i\in
\N}A_i\}$ is upper semi-continuous. Then for any Moore
projection $\Phi$ associated to $\mathcal{U}$, the image $\Phi(A)$ is
totally disconnected.
\end{lema}

\begin{proof}
  Suppose for a contradiction that $\Phi(A)$ is not totally
  disconnected, then there exists a connected component $C\ssq\Phi(A)$
  which has more than one element. Since connected components of $A$
  project to single points, the set $\Phi^{-1}(C)\ssq A$ cannot be
  connected and therefore decomposes into two disjoint relatively
  closed subsets $C_1$ and $C_2$. As a connected component of a
  compact set, $C$ is compact, and the same is true for its preimage
  $\Phi^{-1}(C)$. Hence, both $C_1$ and $C_2$ are compact.

  For any $x \in C$, the continuum $\Phi^{-1}(x)\ssq
  \Phi^{-1}(C)=C_1\cup C_2$ is either completely contained in $C_1$ or
  completely contained in $C_2$. Consequently, the images $\Phi(C_1)$
  and $\Phi(C_2)$ are disjoint. However, this means that $C$
  decomposes into two disjoint compact sets, contradicting its
  connectedness.
\end{proof}

Finally, the following statement will be useful to verify the upper
semi-continuity of decompositions.

\begin{lema}\label{fill3}
  Let $A_n,\ n\in\N$ be a family of compact, connected and bounded
  sets in $\T^2$. If the sets $\filll(A_n)$ are pairwise disjoint and
  $A_n\rightarrow_{\mathcal{H}}\mathcal{A}$, then
  $\filll(A_n)\rightarrow_{\mathcal{H}}\mathcal{A}$.
\end{lema}

\begin{proof}
  For any fixed $\eps>0$, the fact that $A_n\rightarrow_\mathcal{H}
  \mathcal{A}$ implies that, for any fixed $\eps>0$, $\mathcal{A}\ssq
  B_\eps(A_n)\ssq B_\eps(\filll(A_n))$ for sufficiently large $n$.
  Therefore, it suffices to show that conversely $\filll(A_n)\ssq
  B_\eps(\mathcal{A})$ for sufficiently large $n$.

  Suppose for a contradiction that for some $\varepsilon>0$ there is a
  sequence of integers $n_k\nearrow\infty$ such that for each $k\in\N$
  there exists some $x_k\in\filll(A_{n_k})\setminus
  B_{2\eps}(\mathcal{A})$. Since
  $A_n\rightarrow_{\mathcal{H}}\mathcal{A}$, for $k$ large enough, we
  have that $A_{n_k} \subset B_{\eps} (\mathcal{A})$.  By
  Lemma~\ref{fill2}, the point $x_k$ is contained in some disk $D_k$
  with $\partial D_k \ssq A_{n_k}$, and for large $k$ we have
  $\partial D_k \ssq B_\eps(\mathcal{A})$. Let $x_0$ be an
  accumulation point of $\{x_k\}_{k\in\N}$. Then $x_0 \notin
  B_{\eps}(A)$, and further $x_0$ cannot belong to any disk $D_k$
  since these are pairwise disjoint, which follows from the assumption
  that the sets $\filll(A_n)$ are pairwise disjoint combined with
  Lemma~\ref{fill2}. Now, let $k\in \N$ such that $D_k\cap
  B_\eps(x_0)\neq \emptyset$ and $\partial D_k\subset
  B_{\eps}(\mathcal{A})\ssq B_\eps(x_0)^c$. Then the closest point of
  $\cl[D_k]$ to $x_0$ is contained in $B_\eps(x_0)$ and can therefore
  not belong to the boundary of $D_k$, a contradiction.
\end{proof}

We now finish the proof

\begin{proof}[Proof of Addendum~\ref{Cantorextension}] Let
  $\{\Lambda_i\}_{i\in I}:=\conn(\mathcal{M})$. Since $\mathcal{M}$ is
  of type 3, every element in $\{\Lambda_i\}_{i\in I}$ is bounded.
  Now, $\mathcal{U}=\{\filll(\Lambda_i)\}_{i\in I}\cup \{
  \{x\}:x\notin \bigcup_{i\in I}\filll(\Lambda_i) \}$ is a family of
  bounded continua in $\T^2$.  We claim that it is an upper
  semi-continuous decomposition of $\T^2$.

  First let us see that $\mathcal{U}$ is a partition. Suppose for a
  contradiction that $\filll(\Lambda_i)\cap \filll(\Lambda_j)\neq
  \emptyset$ for some $i\neq j\in I$. Then since $\Lambda_i\cap
  \Lambda_j=\emptyset$, $\Lambda_i$ has to be contained in a bounded
  connected component $D$ of $\Lambda_j^c$ or vice versa.  By
  Proposition~\ref{p.classification}, since the compact boundary of
  the unique doubly essential component of $\mathcal{M}^c$ is left
  invariant and it is contained in the minimal set $\mathcal{M}$, by
  minimality, $\mathcal{M}$ equals the boundary of the unique doubly
  essential component. If $\Lambda_i \subset D$, with $D$ a bounded
  complementary domain of $\Lambda_j^c$, then an open neighborhood of
  a point of $\Lambda_i$ does not intersect the doubly essential
  component, even though $\Lambda_j \subset \mathcal{M}$, a
  contradiction.

  To proceed, the elements of $\mathcal{U}$ are non-separating. Hence,
  it remains to check the upper semi-continuity.  For this the only
  non trivial case is when the sequence of elements in $\mathcal{U}$
  is given by elements in $\{\textrm{Fill}(\Lambda_i)\}_{i\in I}$.
  Take a countable subsequence
  $\{\textrm{Fill}(\Lambda_n)\}_{n\in\N}\subset
  \{\textrm{Fill}(\Lambda_i)\}_{i\in I}$ such that
  $\textrm{Fill}(\Lambda_n)\rightarrow_{\mathcal{H}}\mathcal{X}$.
  Passing to a subsequence if necessary, we may assume that the
  $\Lambda_n$ converge to a continuum $\mathcal{Y}\ssq\mathcal{M}$. We
  therefore have $\mathcal{Y}\ssq\Lambda_i$ for some $i\in I$, and by
  Lemma~\ref{fill3} we have that
\begin{equation}
\filll(\Lambda_n)\rightarrow_\mathcal{H}\mathcal{Y} \ssq \Lambda_i\ssq \filll(\Lambda_i).
\end{equation}
Hence, $\mathcal{U}$ is upper semi-continuous.

Now, let $\Phi \colon \T^2\rightarrow \T^2$ be a Moore projection for
$\mathcal{U}$. Given a point $x\in\T^2$, since $\mathcal{U}$ is
preserved by $f$ (i.e for every element $U_i \in\mathcal{U}$ we have
$f(U_i)\in\mathcal{U}$) the set $\Phi\circ f \circ \Phi^{-1}(x)$
contains only one point $y_x\in\T^2$ for every $x\in \T^2$. Define
$\widetilde{f}:\T^2\rightarrow\T^2$ by $f(x)=y_x$.  We claim that
$\widetilde{f}$ is a homeomorphism. Let us prove first the continuity.
For this we take $x\in \T^2$ and fix $\varepsilon>0$. Then, there
exists a neighbourhood $V$ of $\Phi^{-1}(y_x)$ such that
$\Phi(V)\subset B(y_x,\varepsilon)$. Moreover, there exist a
neighbourhood $U$ of $f^{-1}(\Phi^{-1}(y_x))$ such that $f(U)\subset
V$.  On the other hand, since $\Phi^{-1}(x)=f^{-1}(\Phi^{-1}(y_x))$ we
have that there exist $\delta>0$ such that
$\Phi^{-1}(B(x,\delta))\subset U$.  Therefore, given a point $z\in
B(x,\delta)$ the set $\Phi\circ f \circ \Phi^{-1}(z)$ is contained in
$B(y_x,\varepsilon)$. Hence, $\widetilde{f}$ is continuous.  If we
define $\widetilde{f}^{-1}:\T^2\rightarrow\T^2$ such that
$\widetilde{f}^{-1}(x)$ is the unique point in $\Phi\circ f^{-1} \circ
\Phi^{-1}(x)$, we have that $\widetilde{f}^{-1}$ is exactly the
inverse function of $\widetilde{f}$. Moreover, by an analogous
argument as above we have that $\widetilde{f}^{-1}$ is continuous.
Therefore $\widetilde{f}$ is a homeomorphism.

By definition of $\widetilde{f}$, $\Phi \circ f(x)=\widetilde{f} \circ
\Phi(x)$ holds for every $x\in\T^2$. This implies in particular that
$\widetilde{\mathcal{M}}=\Phi(\mathcal{M})$ is minimal for $\widetilde
f$. Thus, we have that $(f,\mathcal{M})$ is an extension of
$(\widetilde{f},\widetilde{\mathcal{M}})$. Moreover, Proposition
\ref{disc} implies that $\Phi(\mathcal{M})$ is totally disconnected
and therefore either a periodic orbit or a Cantor set.
\end{proof}

\section{Special Cases and Applications}\label{applications}

In this section, we consider several special cases of the
classification and provide a number of relations of the possible
minimal sets with other dynamical properties, such as the rotation set
and orbit behaviour.

\subsection{Homeomorphisms homotopic to an Anosov} \label{Anosov}

To prove Corollary~\ref{l.anosov}, recall classical results on
Anosov diffeomorphisms. 

(1) Manning~\cite{manning:1974} showed that any Anosov diffeomorphism
is topologically conjugate to an algebraic Anosov, i.e. an Anosov
induced by a hyperbolic element of $\textrm{SL}(2, \R)$.

(2) Bowen~\cite{Bowen} showed that a minimal set of an algebraic
Anosov diffeomorphism is either a periodic orbit or a Cantor set.

(3) Walters~\cite{walter} provides the existence of a semiconjugacy, homotopic to the identity, between a homeomorphism
homotopic to an Anosov and the underlying Anosov.

\begin{proof}[Proof of Corollary~\ref{l.anosov}]
  If $\mathcal{M}$ would be of type 1 or 2, then by
  Lemma~\ref{l.compact_bounded}, there exists at least one unbounded
  connected component of $\mathcal{M}$, which we denote $\Lambda$. Let
  $h:\T^2\rightarrow\T^2$ be the semiconjugacy between $f$ and $f_A$,
  given by~\cite{walter}. Then $\mathcal{M}':=h(\mathcal{M})$ is a
  minimal set of $f_A$ and is totally disconnected due
  to~\cite{Bowen,manning:1974}. On the other hand, since $h$ is
  continuous and $\Lambda$ is unbounded, $h(\Lambda)$ must be an
  unbounded continuum in $\mathcal{M}'$, a contradiction. Since
  $\mathcal{M}$ can not be the whole torus either, $\mathcal{M}$ has
  to be of type 3.
\end{proof}

\subsection{Non-wandering torus homeomorphisms}

We start with three statements on periodic circloids of non-wandering torus
homeomorphisms.
\begin{lema}[\cite{jaeger:2009b}, Corollary 3.6]
  \label{l.circloid-emptyint} Suppose $C$ is a periodic circloid of a
  non-wandering torus homeomorphism. If $C$ does not contain periodic
  points, then $C$ has empty interior.
\end{lema}
We call a straight line $L\ssq\R^2$ {\em rational}, if it contains infinitely
many rational points. Note that in particular, this implies that the slope of
$L$ is rational.
\begin{lema}[\cite{jaeger:2009b}, Proposition 3.9]
  \label{l.circloid-rational} Suppose $f\in\homtwo$ has a periodic
  circloid. Then the rotation set of $f$ is contained in a rational line.
\end{lema}

Given $f\in\homeo(\torus)$, we say an $f$-invariant continuum
$C\ssq\torus$ is aperiodic if it does not contain a periodic point.
In~\cite{koropecki2009aperiodic}, Koropecki identified annular
continua as the only possible aperiodic invariant proper subcontinua of
non-wandering torus homeomorphisms.
\begin{teor}[\cite{koropecki2009aperiodic}, Theorem 1.1]
  \label{t.koro} Let $S$ be a compact orientable surface and suppose
  $f\in\homeo(S)$ is non-wandering. Then every aperiodic invariant
  proper subcontinuum of $S$ is an annular continuum.
\end{teor}

The following consequence will be crucial in the proof of
Theorem~\ref{teoa.nonwandering} below.

\begin{lema}
  \label{c.koro}
  Suppose $f\in\homeo(\torus)$ is non-wandering. Then every periodic
  unbounded disk $D$ contains a periodic point in its boundary.
\end{lema}

\begin{proof}
  Let $p$ be the period of $D$. As the boundary of an invariant open
  disk, $\partial D$ is an $f^p$-invariant continuum. Suppose for a
  contradiction that $\partial D$ does not contain a periodic point.
  Then it is an annular continuum $A$ by Theorem~\ref{t.koro}.
  However, the complement of an annular continuum in \torus\ is either
  an open annulus $\mathcal{A}$ or the union of a punctured torus
  $\mathcal{T}$ and a bounded disk $\mathcal{D}$. If any of these sets
  contains the unbounded disk $D$, then by connectedness $D$ must have
  further boundary points in the respective set, a contradiction.
\end{proof}

\begin{proof}[Proof of Theorem~\ref{teoa.nonwandering}]
  Let $f\in\homeo(\torus)$ non-wandering and $\M \neq \T^2$ a minimal set for $f$.
  First, assume $\M$ is of type 1. Then $\conn(\M^c)$ cannot contain
  bounded disks since these would have to be wandering by the
  Classification Theorem. Likewise, $\conn(M^c)$ cannot contain an
  unbounded disk, since by Lemma~\ref{c.koro} this unbounded disk has
  to contain periodic points in the boundary, and the boundary belongs
  to \M. Hence $\M=\torus$.  Secondly, suppose $\M$ is of type 2.
  Obviously, the essential annuli cannot be wandering, therefore $\M$
  is equal to the orbit of the boundary of a periodic essential
  circloid $C$. However, by Lemma~\ref{l.circloid-emptyint} the
  interior of $C$ is empty, and it thus follows that $\partial C=C$.

  Finally, suppose $\M$ is of type 3. If $\M$ is a Cantor extension,
  then for any connected component $\Lambda\in\conn(\M)$ the set
  $\filll(\Lambda)$ is a wandering set. Therefore
  $\inte(\filll(\Lambda))=\emptyset$, which means that $\Lambda$ is
  non-separating. If $\M$ is a periodic orbit extension, but not a
  periodic orbit, then every connected component of $\M$ is an
  aperiodic invariant continuum for some iterate of $f$. By Lemma~\ref{c.koro} it is an annular continuum, 
  and by minimality this annular continuum must coincide with its frontiers. By
  Lemma~\ref{l.frontiers} the frontiers are circloids.
\end{proof}

\subsection{Relations with the rotation set}\label{RotationSet}

We now give a proof of the relation of the structure of minimal sets
to the rotation set for homeomorphisms homotopic to the identity.  We
start with the proof of Corollary~\ref{t.non-empty-rotset}, which
states that if the rotation set of $f\in\homtwo$ has non-empty
interior, then any minimal set is of type 3.

\begin{proof}[Proof of Corollary~\ref{t.non-empty-rotset}] Let
$f\in\homtwo$ have lift $F:\R^2\to\R^2$ and assume that the rotation set
$\rho(F)$ has non-empty interior.  Suppose for a contradiction that
$\mathcal{M}$ is a minimal set of $f$ that is not of type 3, that is, there
exists no doubly essential component in its complement. By
Lemma~\ref{l.rational} the existence of an essential component in
$\mathcal{M}^c$ is excluded, so that all connected components are disks.
Now, \cite[Theorem A]{misiurewicz/ziemian:1991} states that for every vector
$\rho\in\inte(\rho(F))$ there exists a minimal set $\mathcal{M}_\rho$ with
rotation vector $\rho$, that is,
\begin{equation}
\nLim \frac{F^n(z)-z}{n}=\rho ~\textup{for~all}~ z\in\pi^{-1}(\mathcal{M}_\rho).
\end{equation}
Fix $\mathcal{M}_\rho$ for some totally irrational vector
$\rho\in\R^2$. Then, since all points in $\mathcal{M}_\rho$ are
recurrent, $\mathcal{M}_\rho$ has to be contained in the orbit of a
periodic disk $D\ssq \mathcal{M}^c$. This implies that there exists a
set $\mathcal{M}'_\rho\ssq D$ which is minimal for $f^p$, where $p$ is
the period of $D$.

Choose a connected component $D_0$ of $\pi^{-1}(D)$, and a lift
$F:\R^2\to\R^2$ of $f$ that leaves $D_0$ invariant. Fix $z\in
\mathcal{M}_\rho'$ with lift $z_0\in\R^2$ and $\delta>0$ such that
$B_\delta(z)\ssq D$. Then irrationality of $\rho$ together with the
recurrence of $z$ implies that there exists a sequence
$\{n_k\}_{k\in\N}$ of integers such that $\kLim f^{n_kp}(z)=z$,
whereas $F^{n_kp}(z_0)$ is unbounded. Consequently, for sufficiently
large $k\in\N$ we have that $F^{n_kp}(z_0)\ssq B_\delta(z_0)+v$ for
some $v\in\Z^2\smin\{0\}$. However, this means that $D_0$ contains
both $z_0$ and $z_0+v$, contradicting the fact that $D$ is
homotopically trivial in $\torus$.
\end{proof}

Finally, we turn to the proof of Corollary~\ref{t.ipr}, which states
that if $f$ is a non-wandering pseudo-rotation with rotation vector
$\rho$ and $\mathcal{M} \neq \T^2$ is a minimal set, then
  \begin{itemize}
  \item[(a)] if $\rho$ is totally irrational, then $\M$ is an extension of a Cantor set, and
  \item[(b)] if $\rho$ is rational, then $\M$ is either an extension
    of a Cantor set, or the periodic orbit of a point or a
    homotopically trivial circloid.
  \end{itemize}

  Given a lift $F:\R^2\to\R^2$ of $f\in\homtwo$, the function
  $\varphi(z)=F(z)-z$ is doubly periodic and can therefore be
  interpreted as a function on the torus.
  \begin{teor}[\cite{matsumoto:2011}, Theorem 2]\label{t.matsumoto} Suppose
  $f\in\homtwo$ is non-wandering and $A$ is an essential annular
  continuum. Further, suppose there exists an $f$-invariant
  probability measure $\mu$ with support in $A$ and $\rho=\int\varphi
  \ d\mu\in\Q^2$.  Then there exists a periodic point in $A$ with
  rotation vector $\rho$.
\end{teor}

\begin{proof}[Proof of Corollary~\ref{t.ipr}] Suppose $f\in\homtwo$ is
  a non-wandering pseudo-rotation with rotation vector $\rho$ and
  $\mathcal{M} \neq \T^2$ a minimal set for $f$.

(a) Let $\rho$ be totally irrational. We have to rule out cases
$1^\textrm{nw}$ and $3^\textrm{nw}$ in
Theorem~\ref{teoa.nonwandering}. First, by
Lemma~\ref{l.circloid-rational}, \M\ cannot be a union of periodic
essential circloids, since these force the rotation set to be included
in a rational line.  Similarly, \M\ cannot be a periodic orbit or a
periodic orbit extension, since this implies the existence of a
rational rotation vector.  Note here that the factor map in the
definition of a periodic orbit extension preserves rotation vectors.

(b) Let $\rho$ be rational. In this case, we have to show that only
cases $2^\textrm{nw}$ and $3^\textrm{nw}$ in
Theorem~\ref{teoa.nonwandering} can occur. However, as a rational
pseudo-rotation $f$ has at least one periodic orbit\cite{franks:1989}.
Thus $\mathcal{M} \neq \T^2$.  Further, due to
Theorem~\ref{t.matsumoto} any periodic essential circloid has to
contain a periodic point. This rules out case $1^\textrm{nw}$.
\end{proof}

\section{Remarks and Problems}\label{Examples}

The results in this paper give rise to a number of further problems to
be elaborated upon. A recurring theme in exploring the structure of
minimal sets is the existence (or not) of unbounded disks.

\begin{problem}[Unbounded disks]
Let $\mathcal{M}$ be a minimal set of a homeomorphism $f \in \homeo(\T^2)$ of type 1 or 2.
\begin{enumerate}
\item[\textup{(i)}] If $\mathcal{M}$ is a type 2 minimal set, is it
  possible to have unbounded disks in the complement of $\mathcal{M}$?
  Note that if if there exists some unbounded disks in
  $\conn(\mathcal{M}^c)$, then there have to be infinitely many, since
  all disks are wandering by Theorem~\ref{TEOA}.
\item[\textup{(ii)}] Do there exist rational pseudo-rotations with
  type 1 minimal sets?

\end{enumerate}
\end{problem}
For recent progress concerning the problem of boundedness of invariant
disks, see \cite{korotal}.

In Corollary~\ref{t.non-empty-rotset} and~\ref{t.ipr}, we considered
the relation between the rotation set of a homeomorphism $f \in
\homeo(\T^2)$ and the structure of the minimal set in specific cases.
In~\cite{F}, a classification was given for the non-resonant case,
i.e. where the rotation set is a single totally irrational vector.
Between the cases given, there is an important class of rotation sets
consisting of line segments.

\begin{problem}[Rotation set versus structure of minimal sets]
  Let $\mathcal{M}$ be a minimal set of a homeomorphism $f \in
  \homeo(\T^2)$. Suppose the rotation set $\rho(f)$ is a line segment
  of positive length. Relate the properties of this line segment with
  the structure of the minimal sets the homeomorphism admits.
\end{problem}

For homeomorphisms homotopic to the identity, all types
of minimal sets that our classification allows are realised,
see~\cite{F} for these constructions. In the case where the
homeomorphism is homotopic to an Anosov, the list of possible minimal
sets is rather restricted, cf. Corollary~\ref{l.anosov}. The case left
is the class of homeomorphisms homotopic to neither the identity, nor
to an Anosov, which in case of the torus are precisely the
Dehn-twists. In this case, examples of type 2 as well as type 3
minimal sets are well-known to occur as minimal sets. Concerning type
1, taking a minimal Dehn-twist and blowing an orbit up to bounded
disks, one obtains minimal sets for which the
complement is a union of bounded disks. This leaves open one case for
Dehn-twists.

\begin{problem}[Unbounded disks and Dehn-twists]
  Is it possible for a homeomorphism homotopic to a Dehn-twist to have
  a type 1 minimal set with either periodic or wandering {\em
    unbounded} disks?
\end{problem}


\begin{thebibliography}{pusus}

\bibitem{auslander}
J.~Auslander.
\newblock {\em Minimal flows and their extensions}.
\newblock Elsevier Science Ltd, 1988.

\bibitem{nakayama}
A.~Bis, H.~Nakayama, and P.~Walczak.
\newblock Locally connected exceptional minimal sets of surface homeomorphisms.
\newblock {\em Ann. de l'Institut Fourier}, 54(3):711--732, 2004.

\bibitem{Bowen} R. Bowen. \newblock Markov partitions and minimal sets for
  Axiom A diffeomorphisms. {\em Amer. J. Math.}  92:903-918, 1970.


\bibitem{franks:1970}
J.~Franks.
\newblock Anosov diffeomorphisms; global analysis.
\newblock {\em Proceedings of the Symposium on Pure Mathematics}, 14:61--93, 1970.

\bibitem{franks:1988}
J.~Franks.
\newblock Generalizations of the {P}oincar\'e-{B}irkhoff theorem.
\newblock {\em Ann. Math. (2)}, 128(1):139--151, 1988.

\bibitem{franks:1989}
J.~Franks.
\newblock Realizing rotation vectors for torus homeomorphisms.
\newblock {\em Trans. Am. Math. Soc.}, 311(1):107--115, 1989.

\bibitem{franks/lecalvez:2003}
J.~Franks and P.~le~Calvez.
\newblock Regions of instability for non-twist maps.
\newblock {\em Ergodic Theory Dyn. Syst.}, 23(1):111--141, 2003.

\bibitem{handel:1982}
M.~Handel.
\newblock A pathological area preserving ${C}^\infty$ diffeomorphism of the
  plane.
\newblock {\em Proc. Am. Math. Soc.}, 86(1):163--168, 1982.

\bibitem{herman:1986}
M.~Herman.
\newblock Construction of some curious diffeomorphisms of the {R}iemann sphere.
\newblock {\em J. Lond. Math. Soc.}, 34:375--384, 1986.

\bibitem{jaeger:2009b}
T.~J\"{a}ger.
\newblock Linearisation of conservative toral homeomorphisms.
\newblock {\em Invent. Math.}, 176(3):601--616, 2009.

\bibitem{katok/hasselblatt:1997}
A.~Katok and B.~Hasselblatt.
\newblock {\em Introduction to the Modern Theory of Dynamical Systems}.
\newblock Cambridge University Press, 1997.

\bibitem{koropecki:2007}
A.~Koropecki.
\newblock {\em On the dynamics of torus homeomorphisms}.
\newblock PhD thesis, IMPA (Brazil), 2007.

\bibitem{koropecki2009aperiodic}
A.~Koropecki.
\newblock Aperiodic invariant continua for surface homeomorphisms.
\newblock {\em Math. Zeitschrift}, 266(1):229--236, 2010.

\bibitem{korotal} A.~Koropecki and F.~Armando Tal. \newblock Strictly
  toral dynamics.  \newblock Preprint 2012, {\tt
    http://arxiv.org/abs/1201.1168}.

\bibitem{F} F.~Kwakkel. \newblock Minimal sets of non-resonant torus homeomorphisms.
\newblock
{\em Fund. Math.}, 211:41-76, 2011.

\bibitem{manning:1974}
A. ~Manning.
\newblock There are no new Anosov diffeomorphisms on tori.
\newblock {\em Amer. J. Math. }, 96(3): 422--429, 1974.

\bibitem{matsumoto:2011}
S.~Matsumoto.
\newblock Rotation sets of invariant separating continua of annular
  homeomorphisms.
\newblock Preprint 2011, {\tt http://arxiv.org/abs/1011.3176v2}. 

\bibitem{MaNa}
S.~Matsumoto and H.~Nakayama.
\newblock Continua as minimal sets of homeomorphisms of $S^2$.
2011, \newblock {\tt http://arxiv.org/abs/1005.0360v2}.

\bibitem{misiurewicz/ziemian:1989}
M.~Misiurewicz and K.~Ziemian.
\newblock Rotation sets for maps of tori.
\newblock {\em J. Lond. Math. Soc.}, 40:490--506, 1989.

\bibitem{misiurewicz/ziemian:1991}
M.~Misiurewicz and K.~Ziemian.
\newblock Rotation sets and ergodic measures for torus homeomorphisms.
\newblock {\em Fund. Math.}, 137(1):45--52, 1991.

\bibitem{shen/yi:1998}
W.~Shen and Y.~Yi.
\newblock Almost automorphic and almost periodic dynamics in skew product
  semiflows.
\newblock {\em Mem.\ Am.\ Math.\ Soc.}, 136(647), 1998.

\bibitem{veech}
W.A. Veech.
\newblock Almost automorphic functions on groups.
\newblock {\em Amer. J. Math.}, 87(3):719--751, 1965.

\bibitem{walter} P. Walters. \newblock Anosov diffeomorphisms are
  topologically stable. \newblock {\em Topology} 9:71--78, 1970.

\bibitem{whyburn}
G.~Whyburn.
\newblock {\em Analytic topology.}
\newblock AMS Colloquium Publications 28, 1942.




\end{thebibliography}
\end{document}